\documentclass{amsart}
\usepackage{amsmath}
\newtheorem{theorem}{Theorem}[section]
\newtheorem{lemma}[theorem]{Lemma}

\theoremstyle{definition}

\newtheorem{definition}[theorem]{Definition}

\theoremstyle{remark}

\newcommand{\fA}{{\mathfrak A}}
\newcommand{\fF}{{\mathfrak F}}
 
\newcommand{\fK}{{\mathfrak K}}
\newcommand{\fN}{{\mathfrak N}} 
\newcommand{\fS}{{\mathfrak S}}
\newcommand{\fU}{{\mathfrak U}}
\newcommand{\cser}{\mathcal{C}}

\newcommand{\ideq}{\unlhd}
\newcommand{\Lie}{{}^{\text{Lie}} \mspace{-1.5mu}{}}
\DeclareMathOperator{\Leib}{Leib}
\DeclareMathOperator{\loc}{Loc}

\title[Locally defined formations]{On locally defined formations of soluble Lie and Leibniz algebras}
\author{Donald W. Barnes}
\address{1 Little Wonga Rd.\\Cremorne NSW 2090\\Australia\\}
\email{donwb@iprimus.com.au}

\subjclass[2010]{Primary 17B30, Secondary  20D10}
\keywords{Lie algebras, Leibniz algebras, saturated formations, local definition}

\begin{document}

\begin{abstract}   It is well-known that all saturated formations of finite soluble groups are locally defined and, except for the trivial formation,  have many different local definitions.  I show that for Lie and Leibniz algebras over a field of characteristic $0$, the formations of all nilpotent algebras and  of all soluble algebras are the only locally defined formations and the latter has many local definitions.   Over a field of non-zero characteristic, a saturated formation of soluble Lie algebras has at most one local definition but a locally defined saturated formation of soluble Leibniz algebras other than that of nilpotent algebras has more than one local definition.
\end{abstract}

\maketitle

\section{Introduction}
In the theory of finite soluble groups, a formation function is a function $f$ which assigns to each prime $p$, a formation $f(p)$ of finite soluble groups. The formation locally defined by $f$ is the class $\loc(f)$ of all finite soluble groups such that, for each chief factor $H/K$ of $G$ with $p$ dividing $|H/K|$,  we have $G/\cser_G(H/K) \in f(p)$.  (Here $\cser_G(H/K)$ denotes the centraliser in $G$ of $H/K$.)  It is a saturated formation.  Apart from the trivial formation which contains only the group of order $1$, every saturated formation of finite soluble groups has many such local definitions.  This theory is set out in Doerk and Hawkes \cite{DH}.  

If $\fK$ is a formation of soluble Lie algebras over the field $F$, the formation locally defined by $\fK$ is the class of Lie algebras $\loc(\fK)$ of those Lie algebras such that, for each chief factor $H/K$ of $L$, we have $L/ \cser_L(H/K) \in \fK$.  It is a saturated formation.  In contrast to the group case, not every saturated formation of soluble Lie algebras is locally defined. The theory of saturated formations of soluble Lie algebras is set out in  Barnes and Gastineau-Hills in \cite{BGH}. The analogous theory for Leibniz algebras is set out in Barnes \cite{SchunckLeib}.

For finite soluble groups, the formation function $f$ is called {\em integrated} if $f(p) \subseteq \loc(f)$ for all primes $p$.  It is called {\em full} if for all $p$, every extension of a $p$-group by a group in $f(p)$ is in $f(p)$.  For Lie and Leibniz algebras, trivially, $\fK \subseteq \loc(\fK)$, thus  the analogue of ``integrated'' always holds.  The significance of the condition ``full'' is that if $N$ is a normal $p$-subgroup of a group $G$ and $V$ is an irreducible $G$-module over the field of $p$ elements, then $N$ acts trivially on $V$.  There is no condition on an ideal $N$ of a Lie or Leibniz algebra $L$ that ensures that $N$ must act trivially on any irreducible $L$-module.  Thus there is no Lie or Leibniz analogue of the condition ``full''.

\section{Lie algebras over a field of characteristic $0$}

The formation $\fN$ of nilpotent Lie algebras is locally defined by the formation $\{0\}$ containing only the zero algebra.  If $\fK$ is a non-zero formation, then it contains the formation $\fA$ of abelian algebras.  But if $H/K$ is a chief factor of any soluble Lie algebra $L$ over a field of characteristic $0$, then $L/\cser_L(H/K)$ is abelian.  Thus for any non-zero formation $\fK$, we have $\loc(\fK) = \fS$, the formation of all soluble Lie algebras.

For Lie algebras in contrast to the case for groups, it is convenient to require that formations be non-empty.  If we allow formations to be empty, then the zero formation $\{0\} = \loc(\emptyset)$ and this is its only local definition.

We thus have, for a field $F$ of characteristic $0$, $\fN$ and $\fS$ are the only locally defined formations, $\fN$ has the unique local definition $\fN= \loc(\{0\})$  while $\fS = \loc(\fK)$ for any non-zero formation $\fK$.  If $F$ is algebraically closed, then $\{0\}, \fN$ and $\fS$ are the only saturated formations, but if $F$ is not algebraically closed, then other saturated formations exist. (See Barnes \cite{SatF0}.)

\section{Lie algebras over a field of characteristic $p \ne 0$}
Let $F$ be a field of characteristic $p \ne 0$.  Then not all saturated formations of soluble Lie algebras over $F$ are locally definable, for example, the formation $\fU$ of supersoluble Lie algebras is saturated but not locally definable.  I show that if the saturated formation $\fF$ is locally defined, then the formation $\fK$ with $\loc(\fK) = \fF$ is uniquely determined by $\fF$.

Denote the nil radical of $L$ by $N(L)$.  The saturated formation locally defined by $\fK$ consists of those algebras $L$ such that $L/N(L) \in \fK$, algebras which are the extension of a nilpotent algebra $N= N(L)$ by an algebra in $\fK$.  In a widely used terminology, they are
nilpotent by $\fK$.  It makes sense to denote the locally defined formation by $\fN \fF$. I reverse this construction.

\begin{definition}\label{def-div} Let $\fF$ be a saturated formation.  We  define
the quotient of $\fF$ by $\fN$ to be
$$\fF/\fN = \{L/A \mid L \in \fF, A \ideq L, N(L) \subseteq A\}.$$ \end{definition}

\begin{lemma} \label{lem-quot} Let $\fF$ be a saturated formation of soluble Lie algebras over any field.  Then
$\fF/\fN$ is a formation. \end{lemma}

\begin{proof} Put $\fK = \fF / \fN$.  By its definition, $\fK$ is quotient closed.  Suppose that $A_i \ideq L$ and that $L/A_i \in \fK$, $i= 1,2$.  We have
to prove that $L/(A_1 \cap A_2) \in \fF$.  We may suppose without loss of
generality that $A_1 \cap A_2 = 0$.

There exist $X_i \in \fF$ with $N(X_i) \subseteq  B_i \ideq X_i$ and $X_i/B_i \simeq L/A_i$.  Then $L/A_i \in \fF$, so $L \in \fF$.  

We have epimorphisms $\phi_i : X_i/B_i \to L/(A_1+A_2)$ and can identify $L$
with $\{(y_1, y_2) \in X_1/B_1 \oplus X_2/B_2 \mid \phi_1(y_1) = \phi_2(y_2)
\}$.  Let $\psi_i$ be the composite of the natural homomorphism $X_i/N(X_i) \to
X_i/B_i$ with $\phi_i$ and let 
$$X = \{(x_1, x_2) \in X_1/N(X_1) \oplus X_2/N(X_2) \mid \psi_1(x_1) =
\psi_2(x_2) \}.$$
Then the natural homomorphism $X_1/N(X_1) \oplus X_2/N(X_2) \to X_1/B_1 \oplus X_2/B_2$ maps the subalgebra  $X$ onto $L$.  As $\fK$ is quotient closed, it
is sufficient to prove that $X \in \fK$.  But this is the special case $B_i = N(X_i)$.

The chief factors of $X_i$ are $\fF$-central irreducible $L$-modules on which $B_1$ and $B_2$ act trivially.  Let $V$ be the direct sum of the chief factors of $X_1$ and $X_2$.  Then the split extension $E$ of $V$ by $L$ is in $\fF$.  Since
the intersection of the centralisers in $X_i$ of the chief factors of $X_i$ is $B_i$, the intersection of the centralisers in $E$ of the chief factors of $E$ is $V$ and $N(E) = V$.  Thus $L \in \fK$.
\end{proof}

The above lemma holds over any field, but the following theorem needs the assumption of non-zero characteristic.

\begin{theorem}  Let $\fF= \loc(\fK)$ be a locally defined formation of soluble Lie algebras over a field $F$ of non-zero characteristic.  Then $\fK = \fF/\fN$.
\end{theorem}

\begin{proof}  Clearly, $\fF/ \fN \subseteq \fK$.  Suppose $L \in \fK$.   By Jacobson \cite[Theorem VI.2, p. 205]{Jac}, there exists a faithful, completely reducible $L$-module $V$.  Let $V = \oplus V_i$ where the $V_i$ are irreducible.   Let $K_i$ be the kernel of the representation of $L$ on $V_i$.  Since $L/K_i \in \fK$, the split extension $E_i$ of $V_i$ by $L/K)i$ is in $\fF$.  The representation of $L/K_i$ on $V_i$ is faithful, so $N(E_i) = V_i$, and $L/K_i \in \fF/\fN$.  But $\cap_i(K_i) = 0$ and it follows that $L \in \fF/\fN$.
\end{proof}

\section{Leibniz algebras}
For Leibniz algebras, the situation is a little more complicated than for Lie algebras.  A left Leibniz algebra is a linear algebra $L$ whose left multiplication operators $d_a : L \to L$ given by $d_a(x) = ax$ for $a,x \in L,$ are derivations.  The basic theory of Leibniz algebras is set out in Loday and Pirashvili \cite{LP}.  The theory of saturated formations of soluble Leibniz algebras is given in Barnes \cite{SchunckLeib}.

The subspace $\Leib(L) = \langle x^2 \mid x \in L \rangle$  is an abelian ideal of $L$ and $L/\Leib(L)$ is a Lie algebra.  Further, $\Leib(L) L =0$.  If $V$ is an irreducible $L$-module, then by a theorem of Loday and Pirashvili \cite{LP}, $L/\cser_L(V)$ is a Lie algebra and $V$ is either symmetric ($vx = -xv$ for $x \in L$ and $v \in V$) or antisymmetric ($VL=0$).  It follows that if $\fF = \loc(\fK)$ is a locally defined formation of soluble Leibniz algebras, then $\fF = \loc(\Lie\fK)$ where $\Lie\fK$ is the class of all Lie algebras in $\fK$.

The characteristic $0$ case is very like that of Lie algebras.  The saturated formation $\fN$ of nilpotent Leibniz algebras is locally defined by the zero formation.  If $\fK$ is a non-zero formation, then it contains the abelian algebras and $\loc(\fK) = \fS$, the class of all soluble Leibniz algebras.  As for Lie algebras, $\fN$ and $\fS$ are the only locally defined formations.

The case of characteristic $p \ne 0$ is closely tied to the corresponding case for Lie algebras.  By Barnes \cite[Corollary 3.17]{SchunckLeib}, if $\fF$ is a saturated formation of soluble Leibniz algebras, then $\Lie\fF$ is a saturated formation of soluble Lie algebras and we have a one to one correspondence $\fF \leftrightarrow \Lie\fF$.  If $\fF = \loc(\fK)$, then both $\fF$ and $\Lie\fF$ are locally defined in the corresponding categories by $\Lie\fK$.  We thus have the following theorem.

\begin{theorem}  Let $\fF$ be a locally defined formation of soluble Leibniz algebras over a field of charactistic $p \ne 0$.  Then there exists a unique formation $\fK_0$ of soluble Lie algebras with $\fF = \loc(\fK_0)$.
\end{theorem}

However, except for the case of $\fN = \loc(\{0\})$,  there are always other formations $\fK$ of soluble Leibniz algebras with $\loc(\fK) = \loc(\fK_0)$.  

\begin{lemma}  Let $\fK_0 \ne \{0\}$ be a formation of soluble Lie algebras.  Let $P_1, \dots, P_n$ be soluble Leibniz algebras such that $P_i/\Leib(P_i) \in \fK_0$ and let $\fK$ be the smallest formation of soluble Leibniz algebras containing $\fK_0$ and the algebras $P_i$.  Then $\Lie\fK = \fK_0$ and $\loc(\fK) = \loc(\fK_0)$.
\end{lemma}

\begin{proof}  We have to show that, if $L \in \fK$, then $L/\Leib(L) \in \fK_0$.   Since $L \in \fK$, there exist ideals $N_1, \dots N_r$ of $L$ with $\cap_i N_i = 0$ such that, for each $i$, either $L/N_i \in \fK_0$ or $L/N_i \simeq P_j$ for some $j$.   We have the natural epimorphisms $\phi_i: L \to L/N_i$ and $L$ may be regarded as the subalgebra of $\oplus_i L/N_i$ whose elements are the $(\phi_1(x), \dots, \phi_r(x))$ for $x \in L$.  But $\phi_i$ maps $x^2$ to $\phi_i(x)^2$, thus maps $\Leib(L)$ onto $\Leib(L/N_i)$ and so gives an epimorphism $\psi_i:L/\Leib(L) \to (L/N_i)/\Leib(L/N_i)$.  This expresses $L/\Leib(L)$ as a subdirect sum of the $(L/N_i)/\Leib(L/N_i)$.  But for all $i$,  $(L/N_i)/\Leib(L/N_i) \in \fK_0$.  As $\psi_i(L/\Leib(L)) \in \fK_0$ for all $i$,  $L/\Leib(L) \in \fK_0$ as asserted.
\end{proof}

We can always construct Leibniz algebras $P$ with $P/\Leib(P) \in \fK_0$.  Let $L$ be any Lie algebra in $\fK_0$ and let $V$ be any non-trivial left $L$-module.  We make this into a Leibniz module by defining $VL = 0$.  Let $P$ be the split extension of $V$ by $L$.  Then $P$ is not a Lie algebra, so not in $\fK_0$.  We have $\Leib(P)=V$ and so $P/\Leib(P) \simeq L \in \fK_0$.  Thus we can always produce $\fK \ne \fK_0$ with $\loc(\fK) = \loc(\fK_0)$.

\bibliographystyle{amsplain}

\end{document}